\documentclass[british,english]{amsart}
\usepackage{amsmath,amsthm,amssymb,mathtools}
\usepackage[mathscr]{eucal}
 \usepackage{cite}
\usepackage{upgreek}
\usepackage[bookmarks,bookmarksnumbered]{hyperref}
\usepackage{enumerate}
\usepackage{amsmath}
\usepackage{mathrsfs}
\usepackage{blindtext}
\usepackage{scrextend}
\usepackage{enumitem}
\addtokomafont{labelinglabel}{\sffamily}
\usepackage{color}
\usepackage[at]{easylist}

\numberwithin{equation}{section}

\theoremstyle{plain}

\newtheorem{theorem}{Theorem}[section]
\newtheorem{claim}[theorem]{Claim}
\newtheorem{lemma}[theorem]{Lemma}
\newtheorem{proposition}[theorem]{Proposition}
\newtheorem{corollary}[theorem]{Corollary}
\theoremstyle{definition}
\newtheorem{definition}[theorem]{Definition}
\newtheorem*{definition*}{Definition}

\newtheorem{remark}[theorem]{Remark}

\newcommand{\T}{\mathbb{T}}
\newcommand{\F}{\mathbb{F}}

\newcommand{\cB}{\mathcal{B}}
\newcommand{\C}{\mathbb{C}}

\newcommand{\cH}{\mathcal{H}}

\newcommand{\cK}{\mathcal{K}}

\newcommand{\N}{\mathbb{N}}

\newcommand{\cU}{\mathcal{U}}

\newcommand{\Z}{\mathbb{Z}}\newcommand{\cZ}{\mathcal{Z}}

\newcommand{\Ad}{\operatorname{Ad}}

\newcommand{\id}{\operatorname{Id}}

\newcommand{\eps}{\varepsilon}

\begin{document}
\title{Characters of surface groups}
\author{David Gao, Adrian Ioana, Itamar Vigdorovich}
\thanks{I.V. was supported by NSF postdoctoral fellowship DMS-2402368. A.I. was supported by NSF grants DMS-2153805 and DMS-2451697}
\begin{abstract}
We initiate the study of characters of surface groups and their corresponding tracial representations. We show that any tracial representation can be approximated arbitrarily well in the Wasserstein topology by factorial tracial representations with spectral gap. In particular, we deduce that the space of traces of a surface group is the Poulsen simplex, thereby resolving positively a question posed by Orovitz, Slutsky, and the third author.
\end{abstract}

\maketitle
\section{Introduction}
Let $\Sigma_g$ denote the closed orientable surface of genus $g \geq 2$, and let $\Gamma_g=\pi_1(\Sigma_g)$ denote its fundamental group. The representation theory of $\Gamma_g$ has been studied in many different contexts. For example, the moduli space of representations $\mathrm{Rep}(\Gamma_g,G)$, where $G$ is a semisimple Lie group, together with the character variety $\mathrm{Char}(\Gamma_g,G)=\mathrm{Rep}(\Gamma_g,G)/G$, has been studied in the very active area of Higher Teichm\"uller Theory; see \cite{wienhard2018invitation} and the references therein. The case where $G$ is an arbitrary compact group is studied in \cite{breuillard2006dense}. The unitary representation theory of $\Gamma_g$, which can be understood in terms of the representation variety in the case $G=\cU(\cH)$ for a fixed separable infinite-dimensional Hilbert space $\cH$, has also been studied. For example, it is shown in \cite{lubotzky2004finite} that the space of representations with finite image is dense in the space of all representations with respect to the Fell topology.

In this paper we focus on \emph{tracial representations}. These are unitary representations $\pi:\Gamma_g\to \cU(\cH)$ such that 
the corresponding algebra
$\pi(\C\Gamma_g)$
admits a tracial positive linear functional (see Section 2 for the exact definition). This class in particular includes all finite-dimensional representations, the regular representation, as well as equivalence-relation representations arising from probability-measure-preserving actions. There are several reasons for taking this approach:
\begin{itemize}
    \item The famous theorem of Thoma \cite{Thoma-characters} states that for a discrete group which is not virtually abelian, the unitary dual (that is, the space of equivalence classes of irreducible unitary representations) is not a standard Borel space, and is thus pathological from the point of view of descriptive set theory. To overcome this, Thoma advocated restricting attention to tracial representations, for which the space of all such representations (up to quasi-equivalence) is a Polish space.
    \item A tracial representation $\pi$ is characterized by its corresponding character,  $\chi_\pi=\mathrm{tr}\circ \pi$, mirroring finite-dimensional representation theory. This yields a parametrization of the moduli space of tracial representations in terms of functions on the group.
    \item The space of unitary representations is usually endowed with the Fell topology, or with the pointwise weak operator topology (WOT), both of which are rather weak and often make density statements less potent. The space of tracial representations, on the other hand, is also naturally endowed with the topology of pointwise convergence of the corresponding characters, and furthermore with the rather strong Wasserstein $2$-metric, allowing for potentially sharper statements.
    \item Tracial representations of fundamental groups of manifolds admit a geometric interpretation: they yield flat Hilbertian bundles that give rise to \(L^2\)-index theory and torsion theory \cite{carey2000correspondences}.
\end{itemize}

We now turn to a more concrete definition of our notions; see Section 2 for the precise details. Let $\Gamma$ be an arbitrary countable group and $\pi:\Gamma\to \cU(\cH)$ a unitary representation. Then $\pi$ extends to a $*$-representation of the group algebra $\pi:\C\Gamma\to \cB(\cH)$. The von Neumann algebra of $\pi$ is the closure
\[
M_\pi=\overline{\pi(\C\Gamma)}^{\mathrm{WOT}}\subset \cB(\cH)
\]
with respect to the weak operator topology on $\cB(\cH)$. Suppose that there exists a faithful positive linear functional $\tau:M_\pi\to \C$ which satisfies $\tau(xy)=\tau(yx)$ for all $x,y\in M_\pi$, and whose restriction to the unit ball is WOT-continuous. In that case we say that the von Neumann algebra $M_\pi$ is \emph{tracial}, and that  $\pi$ is a \emph{tracial representation}. The composition $\chi_\pi=\tau\circ\pi$ is then a \emph{trace} on $\Gamma$, that is, a positive-definite, conjugation-invariant function  on $\Gamma$ normalized by $\chi_\pi(e)=1$, and all such functions arise in this way. There is thus a one-to-one correspondence between tracial representations (up to quasi-equivalence), and traces on $\Gamma$.

The appropriate notion of irreducibility in the context of tracial representations is factoriality. We say that $\pi$ is \emph{factorial} if the center of the algebra $M_\pi$ consists only of the scalar operators $\C\cdot \id_{\cH}$. This is equivalent to saying that the corresponding trace $\chi_\pi$ is a \emph{character}, that is, a trace on $\Gamma$ which cannot be written as a proper convex combination of two distinct traces. In other words, the space of all traces $\mathrm{Tr}(\Gamma)$ is a compact convex subset of $\ell^\infty(\Gamma)$ in the topology of pointwise convergence, and the space of characters $\mathrm{Ch}(\Gamma)$ consists of its extreme points.

Since our results also apply to free groups, we will use $\Sigma_g^*$ to denote a closed orientable surface of genus $g\geq 2$ which is possibly punctured, and by $\Gamma_g^*$ its fundamental group. Thus $\Gamma_g^*$ can refer either to the surface group $\Gamma_g$ or to the free group $F_g$.

\begin{theorem}\label{thm:main-tracial reps}
    Consider the space of tracial representations of $\Gamma_g^*$. The subset of representations which are factorial and non-amenable is dense with respect to the Wasserstein topology.
\end{theorem}

The Wasserstein topology on tracial representations was introduced in \cite{BianeVoiculescu2001}. We recall its definition in Section 2, but note at this point that it is a very  fine topology. 

Amenability is understood here in the sense of Bekka \cite{bekka1990amenable}: a unitary representation $\pi$ is called amenable if $\pi\otimes \bar{\pi}$ admits almost invariant vectors. Non-amenable representations cannot contain finite-dimensional subrepresentations, and moreover have spectral gap, in the sense that $\|\pi(\mu)\|_{\mathrm{op}}<1$, for some symmetric probability measure $\mu$ on the group. 

We now state the result in terms of the traces corresponding to tracial representations. The compact convex set of traces $\mathrm{Tr}(\Gamma)\subset \ell^\infty(\Gamma)$ is always a Choquet simplex, meaning that points are obtained uniquely as barycenters of probability measures on the extreme points (= characters); see Section 2 for more details. Among all Choquet simplices, the Poulsen simplex, introduced in \cite{Poulsen1961}, stands out: it is the unique Choquet simplex, up to affine homeomorphism, with dense extreme points; see \cite{Lindenstrauss1978}. It is in fact the Fra\"\i ss\'e limit of all finite-dimensional simplices \cite{conley2018fraisse}, and as a result enjoys strong universality and symmetry properties. We obtain the following consequence of Theorem \ref{thm:main-tracial reps}.

\begin{corollary}\label{cor:main-poulsen}
    Every trace $\varphi$ on $\Gamma_g^*$ is a pointwise limit of characters, $\varphi=\lim \chi_n$. Thus, the space of traces $\mathrm{Tr}(\Gamma_g^*)$ is the Poulsen simplex.
\end{corollary}
This answers a question raised in \cite{OrovitzSlutskyVigdorovich2023}. It was shown in \cite{OrovitzSlutskyVigdorovich2023} that the space of traces of a non-abelian free group is the Poulsen simplex. Since one has a natural quotient $\Gamma_g\to F_g$, obtained by puncturing $\Sigma_g$, it follows that the Poulsen simplex of traces on the free group appears naturally as a face of the simplex of traces of the surface group. It was therefore asked whether the latter simplex is itself Poulsen. Corollary \ref{cor:main-poulsen} answers this question positively. In fact, it does so in a stronger sense: the approximation holds in the finer Wasserstein topology, and moreover the approximating sequence $\chi_n$ can be chosen to consist of non-amenable characters. In particular, these are not characters of finite-dimensional representations, and they moreover satisfy \cite[Equation~4.14]{levit2023spectral}.  In contrast, there exist traces on $\Gamma_g^*$ which cannot be approximated pointwise by finite-dimensional traces. Indeed, this follows from the refutation of Connes' embedding problem in \cite{MIP*=RE}.\footnote{We thank M. Chapman for explaining to us how to deduce this statement for surface groups.}

The aforementioned results of \cite{OrovitzSlutskyVigdorovich2023} were generalized in full extent to free products in \cite{ioana2024trace}.  Together with our present results, this might suggest that the property of having a Poulsen trace simplex is abundant among groups with hyperbolic features, but this is far from being the case. Indeed, virtually free groups typically do not have this property; see \cite{OrovitzSlutskyVigdorovich2023,ioana2024trace}. Moreover, random groups in the relevant regime of Gromov's model are hyperbolic and have Kazhdan's property (T). However, property (T) implies that the simplex of traces is a \emph{Bauer simplex}, that is, its extreme points form a closed set \cite{levit2023spectral}, and is therefore as far as possible from being Poulsen.

It is also interesting to understand the space of traces for fundamental groups of closed hyperbolic manifolds of dimension greater than \(2\). Such groups often admit free quotients \cite{lubotzky1996free}. Even more intriguing is the case of complex hyperbolic manifolds, where  more rigidity is expected. The quaternionic and octonionic cases already have property (T), and therefore a Bauer simplex of traces, although one may still expect a rich character theory because of the abundance of normal subgroups. Finally, lattices in simple Lie groups of rank at least \(2\) not only have property (T), but in fact admit a very rigid character theory \cite{Bekka,peterson2015character,bader2023charmenability}. Corollary \ref{cor:main-poulsen} thus stands in strong contrast with the case of lattices in higher-rank Lie groups, and aligns with the familiar rank-one versus higher-rank dichotomy.

Returning to general unitary representations, it follows from \cite{lubotzky2004finite} (and more specifically from the fact that $C^*(\Gamma_g^*)$ is residually-finite-dimensional) that any unitary representation can be approximated arbitrarily well by a tracial representation in the Fell topology.
We thus get the following application of our results to ordinary unitary representations:

\begin{corollary}
    In the space of all equivalence classes of unitary representations, the set of factorial, tracial, non-amenable representations is dense in the Fell topology.
\end{corollary}

The main difficulty arising in the present paper, which does not arise in \cite{OrovitzSlutskyVigdorovich2023,ioana2024trace} is, of course, that a surface group is not a free product. However, we strongly exploit the amalgamated free product structure
\begin{align}\label{eq:surface group amalgam}
\Gamma_g\cong F_2*_\Z F_{2g-2},
\end{align}
where $\Z$ is generated by the product of commutators in each of the two components. 
To that end, we develop three methods of perturbation.

The first method is suitable for any group $\Gamma$ admitting a non-abelian free group as a quotient $q:\Gamma\to F_r$. The idea is to define an ``interpolated regular representation'' $\lambda_\eps$ of the free group, which is arbitrarily close to the trivial representation, and yet shares some regularity properties with the regular representation $\lambda$. A general representation $\pi$ of $\Gamma$ is then perturbed by tensoring:
\[
\pi_\eps=\pi\otimes \lambda_\eps.
\]
This alone allows us to obtain:

\begin{proposition}\label{first_method}
    Let $\Gamma$ be a countable group admitting a non-abelian free quotient. Then the space of non-amenable tracial representations is dense in the space of all tracial representations with respect to the Wasserstein topology.
\end{proposition}


The second method, which we call free conjugation, arises naturally in the context of deformations of amalgamated free products in deformation/rigidity theory \cite{IPP08}, as well as in free liberation processes in free probability theory.
Roughly speaking, one enlarges the von Neumann algebra $M_\pi$ of the given representation $\pi$ to $M_\pi*L^\infty(\T)$ and perturbs part of $M_\pi$ by conjugation with unitaries $u\in L^\infty(\T)$.
In our case, one must  consider the amalgamated free product
\[
M_\pi *_N (N\bar{\otimes} L^\infty(\T)),
\]
where $N$ is the von Neumann algebra of the subgroup $\Z$ appearing in (\ref{eq:surface group amalgam}). This method was also used in \cite{ioana2024trace}, but the presence of the subgroup $\Z$ in the amalgamated free product decomposition of $\Gamma_g$ creates substantial additional difficulties. In technical terms, one must ensure that the von Neumann algebras $\pi(F_2)''$ and $\pi(F_{2g-2})''$ generated by the components in (\ref{eq:surface group amalgam}) are not intertwined into $\pi(\Z)''$ inside $\pi(\Gamma_g)''$ in the sense of Popa. This is part of what makes the first perturbation method (Proposition \ref{first_method})  useful: it removes amenable (and thus type I) summands.

However, even after overcoming this obstacle, the second perturbation method still has a fundamental shortcoming. Namely, we wish to perturb $\pi$ so as to make it factorial, that is, to eliminate the center of $\cZ(M_\pi)$. However, the intersection  \(\cZ(M_\pi)\cap \pi(\Z)''\) is stubborn, in the sense that it remains unchanged under the free conjugation perturbation.

To address this, we develop a third perturbation method which is tailored in a strong way to the defining relation of the surface group. Here one again considers an algebra of the form $M*_{N}(N\bar\otimes L^\infty(\T))$, but now $N$ is the von Neumann algebra generated by a single generator $a$ of a free group $F_2=\langle a,b\rangle$. Moreover, rather than conjugating by $u\in L^\infty(\T)$, one considers right multiplication of $b$ by $u$. The key point is that, by construction, $u$ commutes with $a$, and this gives rise to the commutator identity
$
[a,bu]=[a,b].
$
This allows one to perturb the representation on one of the free groups appearing in the amalgamated free product structure (\ref{eq:surface group amalgam}) without interfering with the other.

The proof of Theorem \ref{thm:main-tracial reps} is obtained by combining these methods. More specifically, we apply first the first perturbation, then the third, then again the first, and finally the second. It is nevertheless possible that the first perturbation is more powerful than we are presently able to show. In other words, it may already possess all the desired properties, so that after applying it to an arbitrary representation, and then applying the second perturbation, one obtains approximating representations with all the required features. Connected to this is the question of whether there is a larger class of amalgamated free products for which the trace space is the Poulsen simplex. 

Finally, a broad and interesting question is which groups admit `meaningful' actions on the Poulsen simplex---say, actions which at the very least do not factor through a finite quotient. Our results imply that the mapping class group $\mathrm{MCG}(\Sigma_g^*)$, identified with the outer automorphism group of $\Gamma_g^*$, acts on the Poulsen simplex $\mathrm{Tr}(\Gamma_g^*)$. In the context of classical and higher Teichm\"uller theory, dynamical systems of this kind have been studied extensively. Natural questions about this action include the study of the \emph{characteristic traces}
$\mathrm{Tr}(\Gamma_g^*)^{\mathrm{MCG}(\Sigma_g^*)}$,
as well as the study of representations of the semidirect product $\mathrm{Aut}(\Gamma_g^*) \ltimes \Gamma_g^*$.

\subsection*{Acknowledgements} The second and third authors would like to thank Pieter Spaas for valuable discussions concerning the problem solved here.

\section{Preliminaries}

Let $\cH$ be a Hilbert space, and let $\cB(\cH)$ denote the algebra of bounded operators on $\cH$. It is endowed with the $*$-operation given by the adjoint, and with the weak operator topology (WOT), defined as follows: a net of bounded operators $T_\alpha$ converges to $T$ in the WOT if and only if
\[
\forall \xi,\eta \in \cH,\qquad \langle T_\alpha\xi,\eta\rangle \to \langle T\xi ,\eta\rangle.
\]
A unital $*$-subalgebra $M\subset \cB(\cH)$ is called a \emph{von Neumann algebra} if it is closed in the weak operator topology. We denote by $\cU(M)$ the group of unitary elements in $M$.

The center of $M$ is
\[
\cZ(M):=\{x\in M:xy=yx\ \text{for all }y\in M\}.
\]
One always has $\C\id_{\cH}\subset \cZ(M)$, and when equality holds we say that $M$ is a \emph{factor}. Equivalently, $M$ has no non-trivial two-sided ideals which are closed in the weak operator topology.

Given a subset $S\subset \cB(\cH)$, its commutant is
\[
S':=\{x\in \cB(\cH):xs=sx\ \text{for all }s\in S\}.
\]
We also write $S'':=(S')'$. Von Neumann's double commutant theorem states that if $S=S^*$ and $\id_{\cH}\in S$, then $S''$ is precisely the smallest von Neumann algebra containing $S$, or equivalently, the WOT-closure of the unital $*$-algebra generated by $S$.

Now let $\Gamma$ be a countable group, and let $\pi:\Gamma\to \cU(\cH)$ be a unitary representation. The associated von Neumann algebra is
\[
M_\pi:=\pi(\Gamma)''=\overline{\pi(\C\Gamma)}^{\mathrm{WOT}}\subset \cB(\cH).
\]
We say that $\pi$ is \emph{factorial} if $M_\pi$ is a factor.

An important example is the left regular representation
\[
\lambda_\Gamma:\Gamma\to \cU(\ell^2(\Gamma)),
\qquad
(\lambda_\Gamma(g)\xi)(h)=\xi(g^{-1}h).
\]
Its associated von Neumann algebra
\[
L\Gamma:=M_{\lambda_\Gamma}
\]
is called the \emph{group von Neumann algebra} of $\Gamma$. It is a factor if and only if $\Gamma$ is ICC, that is, if every non-trivial conjugacy class is infinite. In particular, $LF_r$ is a factor for every free group $F_r$ with $r\geq 2$.

\medskip

A \emph{trace} on a von Neumann algebra $M$ is a linear functional $\tau:M\to \C$ satisfying:
\begin{enumerate}
    \item \emph{normalized:} $\tau(1)=1$;
    \item \emph{positive:} $\tau(x^*x)\geq 0$ for all $x\in M$;
    \item \emph{tracial:} $\tau(xy)=\tau(yx)$ for all $x,y\in M$.
\end{enumerate}
One is typically interested in traces which are moreover:
\begin{enumerate}
    \setcounter{enumi}{3}
    \item \emph{faithful:} if $\tau(x^*x)=0$, then $x=0$;
    \item \emph{normal:} $\tau$ is ultraweakly continuous, or equivalently, WOT-continuous on the (operator norm) unit ball of $M$.
\end{enumerate}

A \emph{tracial von Neumann algebra} is a pair $(M,\tau)$ where $M$ is a von Neumann algebra and $\tau$ is a normal faithful trace. We will often suppress $\tau$ from the notation when it is clear from the context. In the factorial case, whenever such a trace exists, it is unique.

For example, $\cB(\cH)$ admits a trace if and only if $\cH$ is finite-dimensional, in which case necessarily
\[
\tau=\frac{1}{\dim(\cH)}\mathrm{Tr},
\]
where $\mathrm{Tr}$ is the usual matrix trace. Another important example is the group von Neumann algebra $L\Gamma$, which is always tracial, with canonical trace
\[
\tau\!\left(\sum_{g\in \Gamma} a_g u_g\right)=a_e,
\]
where $u_g:=\lambda_\Gamma(g)$. Under the identification $L^2(L\Gamma)\cong \ell^2(\Gamma)$, the elements $\{u_g\}_{g\in\Gamma}$ form an orthonormal basis. Thus every element $x\in L\Gamma$ admits a Fourier expansion
\[
x=\sum_{g\in \Gamma} a_g u_g,
\qquad a_g\in \C,\qquad \sum_{g\in \Gamma}|a_g|^2<\infty.
\]
For example, if $\Gamma=\Z$, then $L\Gamma\cong L^\infty(\T,m_{\mathrm{Leb}})$ where $\T\subset \C$ denotes the unit circle,  and the above is the ordinary Fourier expansion of an element of $L^\infty(\T,m_{\mathrm{Leb}})$.

A tracial von Neumann algebra $(M,\tau)$ is endowed with the inner product
$
\langle x,y\rangle_\tau:=\tau(y^*x),
$
and hence with the $2$-norm
\[
\|x\|_2:=\langle x,x\rangle_\tau^{1/2}=\tau(x^*x)^{1/2}.
\]
The completion of $M$ with respect to this norm is the Hilbert space denoted by $L^2(M,\tau)$, or simply $L^2(M)$. The algebra $M$ acts on $L^2(M)$ by left multiplication, so that we may view
$M\subset \cB(L^2(M))$

If $(M,\tau)$ is a tracial von Neumann algebra, then
\[
L^2(L\Gamma\bar\otimes M)\cong \ell^2(\Gamma)\otimes L^2(M),
\]
and every element $x\in L\Gamma\bar\otimes M$ admits a unique Fourier expansion
\[
x=\sum_{g\in \Gamma} u_g\otimes a_g,
\qquad a_g\in M,\qquad \sum_{g\in \Gamma}\|a_g\|_2^2<\infty.
\]

Given a closed subspace $\cK\subset \cH$ of a Hilbert space, we write $
\cH\ominus \cK$
for the orthogonal complement of $\cK$ in $\cH$. In particular, if $N\subset M$ are tracial von Neumann algebras, then
$
L^2(M)\ominus L^2(N)
$
denotes the orthogonal complement of $L^2(N)$ inside $L^2(M)$.

A von Neumann algebra $M\subset\cB(\cH)$ is called \emph{diffuse} if it has no minimal projections. Assuming $\cH$ is separable (which will be the case throughout), we say that $M$ is of \emph{type I} if it is isomorphic to a direct sum of algebras of the form $L^\infty(X_i, \mu_i) \bar\otimes \cB(\cH_i)$. More generally, if $\{z_i\}_{i\in I}\subset \cZ(M)$ are central projections with $\sum_i z_i=1$, then $M$ decomposes as the direct sum
\[
M=\bigoplus_{i\in I} Mz_i.
\]
We say that $M$ has \emph{no type I direct summand} if there is no nonzero central projection $z\in \cZ(M)$ such that $Mz$ is of type I.

\medskip

Besides tensor products, one also has the notion of free  products. Given tracial von Neumann algebras $(M_i,\tau_i)$, write
\[
M_i^\circ:=L^2(M_i)\ominus \C 1.
\]
Their free product $M_1*M_2$ is realized on the Fock space
\[
L^2(M_1*M_2)
=
\C 1
\oplus
\bigoplus_{k\geq 1}
\bigoplus_{i_1\neq \cdots \neq i_k}
M_{i_1}^\circ \otimes \cdots \otimes M_{i_k}^\circ .
\]
This extends the notion of free products of groups in the sense that
\[
L(\Gamma_1*\Gamma_2)\cong L\Gamma_1 * L\Gamma_2.
\]

Two subalgebras \(P_1,P_2\) of a tracial von Neumann algebra \((M,\tau)\) are called \emph{freely independent} if
\[
\tau(x_1x_2\cdots x_k)=0
\]
whenever $k\geq 1$, \(x_j\in P_{i_j}\ominus \C 1\) and \(i_j\neq i_{j+1}\) for all \(j\). That is, alternating centered words have trace zero.

Likewise, one has the notion of amalgamated free products (see \cite{Po93,VDN92})
\[
M_1 *_N M_2,
\] where $(M_i,\tau_i)$ are tracial von Neumann algebras with a common subalgebra $N$ such that ${\tau_1}_{|N}={\tau_2}_{|N}$.
This construction parallels amalgamated free products of groups: $L(\Gamma_1*_{\Lambda}\Gamma_2)\cong L(\Gamma_1)*_{L(\Lambda)}L(\Gamma_2)$.
 Whenever \(N\subset M\) are tracial von Neumann algebras, the orthogonal projection
\[
L^2(M)\to L^2(N)
\]
gives rise to a map
\[
E_N:M\to N,
\]
called the \emph{conditional expectation}. Writing
\[
M_i^\circ:=L^2(M_i)\ominus L^2(N),
\]
the amalgamated free product \(M_1 *_N M_2\) is built from alternating tensors over \(N\) of the spaces \(M_i^\circ\). In particular, \(M_1 *_N M_2\) is the WOT-closure of the space 
spanned by \(N\) together with reduced words
\[
x_1x_2\cdots x_k,
\]
where \(x_j\in M_{i_j}\), \(E_N(x_j)=0\), and adjacent letters come from alternating algebras, that is, \(i_j\neq i_{j+1}\) for all \(j\).

For further reference, we record the following elementary lemma:

\begin{lemma}\label{intersect}
Let $M=M_1*_NM_2$ be as above. Let $u\in M_2$ is a unitary such that $u-E_N(u)$ is invertible (e.g., if $\|E_N(u)\|<1$). Then $M_1\cap uM_1u^{-1}\subset N$.

\end{lemma}

\begin{proof}
    Let $x\in M_1\cap uM_1u^{-1}$ and put $y=u^{-1}xu$. Then $y\in M_1$ and $xu=uy$. By applying $E_{M_1}$ and using that $E_{M_1}(u)=E_N(u)$, we get that $xE_N(u)=E_N(u)y$, hence $x(u-E_N(u))=(u-E_N(u))y$. By applying $E_{M_2}$ and using that $E_{M_2}(z)=E_N(z)$, for every $z\in M_1$, we also get that $E_N(x)(u-E_N(u))=(u-E_N(u))E_N(y)$ and thus $(x-E_N(x))(u-E_N(u))=(u-E_N(u))(y-E_N(y))$.

    By the definition of amalgamated free products we have that $(M_1\ominus N)(M_2\ominus N)$ is orthogonal to $(M_2\ominus N)(M_1\ominus N)$.
 Hence, $(x-E_N(x))(u-E_N(u))=0$.   
Since $u-E_N(u)$ is invertible, we get that $x-E_N(x)=0$ and so $x\in N$, as claimed.
\end{proof}

A unitary $u$ in a tracial von Neumann algebra $(M,\tau)$ is called a \emph{Haar unitary} if
\[
\tau(u^n)=0 \qquad \text{for all } n\in \Z\setminus\{0\}.
\]
Equivalently, the spectral measure of $u$ with respect to $\tau$ is the Haar measure on the circle $\T$.

More generally, if $x\in M$ is normal, its \emph{spectral measure} $\mu_x$ is the unique Borel probability measure on the spectrum $\sigma(x)$ such that
\[
\tau(f(x))=\int_{\sigma(x)} f(z)\,d\mu_x(z)
\]
for every continuous function $f$ on $\sigma(x)$.

\medskip

We now return to representations. A unitary representation $\pi:\Gamma\to \cU(\cH)$ is called \emph{tracial} if its von Neumann algebra $M_\pi$ admits a normal faithful trace $\tau$. In that case, the function
\[
\chi_\pi:\Gamma\to \C,\qquad \chi_\pi(g):=\tau(\pi(g)),
\]
is a normalized positive-definite class function on $\Gamma$,  called the \emph{trace} associated with $\pi$. Conversely, every such function arises in this way via the GNS construction. Thus, tracial representations, up to quasi-equivalence, correspond to traces on $\Gamma$.

We denote by $\mathrm{Tr}(\Gamma)$ the space of all traces on $\Gamma$. It is a compact convex subset of $\ell^\infty(\Gamma)$ in the topology of pointwise convergence. Its extreme points are called \emph{characters}, and we denote the set of characters by $\mathrm{Ch}(\Gamma)$. A tracial representation is factorial if and only if its associated trace is a character.

Moreover, $\mathrm{Tr}(\Gamma)$ is a Choquet simplex. This means that every trace $\varphi\in \mathrm{Tr}(\Gamma)$ admits a unique representing probability measure, $\mu_\varphi$, on the set of extreme points $\mathrm{Ch}(\Gamma)$, in the sense that
\[
\varphi(g)=\int_{\mathrm{Ch}(\Gamma)} \chi(g)\,d\mu_\varphi(\chi)
\qquad \text{for all } g\in \Gamma.
\]
In other words, every trace decomposes uniquely as a barycenter of characters. We refer to \cite{BlackadarRordam2024} for a concise proof of this standard fact, and to \cite{Phelps2001} for the general theory of Choquet simplices.

Finally, let us briefly recall the Wasserstein topology. Assume for simplicity that $\Gamma$ is finitely generated, and fix a finite generating set $S=\{s_1,\dots,s_d\}$. Given two tracial representations $\pi,\rho$ of $\Gamma$, one defines their Wasserstein distance by
\[
W_2(\pi,\rho)
:=
\inf
\left(
\sum_{j=1}^d \|u_j-v_j\|_2^2
\right)^{1/2},
\]
where the infimum is taken over all realizations of $\pi$ and $\rho$ inside a common tracial von Neumann algebra $(M,\tau)$, with
\[
u_j=\pi(s_j),\qquad v_j=\rho(s_j).
\]
Equivalently, this defines a metric on traces on $\Gamma$. The resulting topology is called the \emph{Wasserstein topology}, introduced in \cite{BianeVoiculescu2001}. It is in general strictly stronger than the topology of pointwise convergence on $\mathrm{Tr}(\Gamma)$; in fact, it is often not separable \cite{gangbo2022duality}.

\section{Three perturbation methods}

In the course of proving Theorem \ref{thm:main-tracial reps}, we developed several perturbation methods, each with its own advantages. It is the combination of all three, used in a particular way, that allows us to obtain the final result. In this section, we introduce each of these methods and discuss their main properties.

\subsection{Perturbation I: interpolated regular representation tensoring}\label{subsection 3.1}

Here we develop a perturbation suitable for groups having a non-abelian free group as a quotient. 

\subsubsection{Quasi-Haar unitaries}
Recall that a unitary $x$ in a tracial von Neumann algebra $(M,\tau)$ is called a Haar unitary if $\tau(x^n)$ is $0$ for all $n\ne 0$ and $1$ otherwise. Equivalently, the spectral distribution $\mu_x\in \mathrm {Prob}(\T)$ given by $\mu_x(f)=\tau(f(x))$ is the Lebesgue measure on $\T\subset \C$.
\begin{definition}
    A unitary element $x$ in a tracial von Neumann algebra $(M,\tau)$ is said to be a \emph{quasi-Haar unitary} if its spectral measure $\mu_x$ is equivalent to (i.e., has the same null sets as) the Lebesgue measure on the circle.
\end{definition}

\begin{lemma}\label{lem:choice_of_gen}
    Let $x$ be a Haar unitary in a tracial von Neumann algebra $(M,\tau)$. For any $\eps > 0$, there exists $y\in M$ such that
    \begin{enumerate}
        \item $y$ is a quasi-Haar unitary.
        \item $\|y-1\|_2<\eps$.
        \item There exists an integer $m$ such that $y^m = x^{m^2}$. 
    \end{enumerate}
\end{lemma}

\begin{proof}
    Let $\eps > 0$. Fix a positive integer $m$ s.t. $\frac{4\pi^2}{m^2} + \frac{4}{m} < \eps^2$.  The von Neumann algebra generated by $x$ is identified in a trace-preserving manner with $L^\infty(\T,\mu)$ where $\mu$ is the Haar measure on $\T$, such that $x$ is mapped to the identity map $\T\to \T$. Thus, the statement boils down to constructing a function $f:\T\to \T$ s.t.
    \begin{enumerate}
        \item $f_\ast(\mu)$ is equivalent to $\mu$.
        \item $\int|f(z) - 1|^2\,d\mu=\int |z - 1|^2 \, f_\ast(d\mu) < \eps^2$.
        \item For some positive integer $m$, $f(z)^m = z^{m^2}$ for all $z \in \T$.
    \end{enumerate}

    For, after we obtain such an $f$, we can simply set $y = f(x)$. Now, for $0 \leq \theta < 2\pi$, we define
    \begin{equation*}
        f(e^{i\theta}) = \begin{cases}
            e^{im\left(\theta - \frac{2k\pi}{m^2}\right)} &,\text{ if }\frac{2k\pi}{m^2} \leq \theta < \frac{2(k + 1)\pi}{m^2}\text{ for some integer }0 \leq k < m^2 - m\\
            e^{im\left(\theta - \frac{2(m - 1)\pi}{m}\right)} &,\text{ otherwise}
        \end{cases}.
    \end{equation*}
    It is straightforward to check that this indeed satisfies the desired requirements on $f$.
\end{proof}

We therefore have:
\begin{lemma}\label{lem:interporlated regular representations}
    Let $F_r=\langle g_1 ,...,g_r\rangle$ be the free group on $r\geq 1$ generators. Then for any $\eps>0$ there exists a representation $\lambda_\eps:F_r\to \cU(M) = \cU(LF_r)$ satisfying:
    \begin{enumerate}
        \item Each of the unitaries $\lambda_\eps(g_i)$ is quasi-Haar.
        \item $\|\lambda_\eps(g_i)-1\|_2<\eps$ for each $i=1,...,r$.
        \item There exists $m\in \N$ such that  $\lambda_\eps(g_i^m)=\lambda(g_i)^{m^2}$ for each $1\leq i\leq r$. 
    \end{enumerate}
\end{lemma}

\begin{proof}
    Let $g_1,...,g_r$ denote the standard generators. Let $\lambda:F_r\to LF_r$ denote the regular representation, and set $x_i=\lambda(g_i)$. Let $y_i\in LF_r$ be as provided in the previous lemma, and set $\lambda_\eps:g_i\mapsto y_i$. 
\end{proof}

\subsubsection{Amenable representations}

We recall Bekka's notion of amenability of unitary representations.

\begin{lemma}\label{lem:bekka_amenable_rep}
For a unitary representation $\pi:\Gamma\to \cU(\cH)$ of a group $\Gamma$, the following are equivalent:
\begin{enumerate}
    \item $\pi$ is \emph{amenable}: there exists a state on $\cB(\cH)$ that is invariant under the adjoint action
    \[
        \Ad(\pi(g)): \cB(\cH)\to \cB(\cH),\qquad T\mapsto \pi(g)T\pi(g)^* .
    \]
    \item $\pi\otimes \bar{\pi}$ has almost invariant vectors.
    \item $\pi\otimes \sigma$ has almost invariant vectors for some unitary representation $\sigma$ of $\Gamma$.
\end{enumerate}
\end{lemma}

\begin{proof}
The equivalence between (1) and (2) is proved in \cite{bekka1990amenable}.
The implication \emph{(2)$\Rightarrow$(3)} is immediate, and \emph{(3)$\Rightarrow$(2)} is proved in \cite[Lemma~3.2]{popa2008superrigidity}.
\end{proof}

\begin{lemma}\label{lem:tensor_permanence}
Let $\sigma$ be a unitary representation of $\Gamma$.
If $\pi\otimes \sigma$ is amenable for some unitary representation $\pi$, then $\sigma$ is amenable.
Equivalently, if $\sigma$ is non-amenable then $\pi\otimes \sigma$ is non-amenable for every $\pi$.
\end{lemma}

\begin{proof}
Assume $\pi\otimes \sigma$ is amenable where $\pi: \Gamma \to \cU(\cK)$ and $\sigma: \Gamma \to \cU(\cH)$. Then there is a state on $\cB(\cK \otimes \cH) = \cB(\cK) \bar\otimes \cB(\cH)$ invariant under the adjoint action of $\Gamma$ induced by $\pi \otimes \sigma$. Restricting to $\cB(\cH) = 1 \otimes \cB(\cH) \subset \cB(\cK) \bar\otimes \cB(\cH)$ yields a state on $\cB(\cH)$ invariant under the adjoint action of $\Gamma$ induced by $\sigma$. Thus, $\sigma$ is amenable.
\end{proof}

We say that a trace $\varphi$ on $\Gamma$ is \emph{amenable} if the corresponding tracial representation  $\pi_\varphi$ is amenable in the sense of
Lemma~\ref{lem:bekka_amenable_rep}.
Note that this is not equivalent to requiring amenability of the von Neumann algebra generated by $\pi_\varphi(\Gamma)$; the latter
requirement corresponds to \emph{uniform} amenability of the trace.

\begin{lemma}\label{lem:no_amenable_summand}
Let $\pi$ be a tracial representation with corresponding von Neumann algebra $(M,\tau)$ and trace $\varphi=\tau\circ \pi $.
If $\varphi$ is non-amenable, then $M$ has no amenable direct summand.
In particular, $M$ has no type~I direct summand.
\end{lemma}

\begin{proof}
If $(M,\tau)$ is an amenable tracial von Neumann algebra and $\pi:\Gamma\to \cU(M)$ is a unitary representation,
then $\pi$ is amenable. For example, amenability of $M$ is equivalent to the existence of almost central, almost tracial, vectors
in the coarse $M$-bimodule $L^2(M)\,\bar\otimes\, L^2(\bar M)$. These vectors are in particular almost invariant for
$\pi\otimes \bar\pi$, hence $\pi$ is amenable by Lemma~\ref{lem:bekka_amenable_rep}.

Now suppose $M = M_1\oplus M_2$, with $M_1$ amenable. Then $\pi_\varphi=\pi_1\oplus\pi_2$ with
$\pi_i:\Gamma\to \cU(M_i)$ tracial representations. By the previous paragraph, $\pi_1$ is amenable, hence
$\pi_\varphi$ is amenable, contradicting that $\varphi$ is non-amenable.
\end{proof}
\subsubsection{Tensoring}

Let $\Gamma$ be a group admitting a free quotient $\phi:\Gamma\to F_r$ with $r\geq 2$. Let $\eps>0$, and fix a representation $\lambda_\eps:F_r\to \cU(LF_r)$ as provided by Lemma \ref{lem:interporlated regular representations}. Denote
\[
\tilde{\lambda}=\lambda\circ \phi
\qquad\text{and}\qquad
\tilde{\lambda}_\eps=\lambda_\eps\circ \phi.
\]
Given a tracial representation $\pi:\Gamma\to \cU(M)$, define
\[
M_\eps= M\bar\otimes LF_r,\qquad \pi_\eps=\pi \otimes \tilde{\lambda}_\eps:\Gamma\to  \cU(M_\eps).
\]
Of course, this construction is not canonical: it depends on the choice of the free quotient, as well as on the choice of the representation $\lambda_\eps$ of that quotient, for which we constructed one suitable example. Nevertheless, once this data is fixed, we regard $\pi_\eps$ as depending only on $\pi$ and $\eps$.

We now establish a few facts about $\pi_\eps$.

\begin{proposition}\label{prop:free_quotient_tensor_perturb}
The representation $\pi_\eps$ is non-amenable, and moreover
\[
\lim_{\eps\to 0}\|\pi_\eps(g)-\pi(g)\|_2=0
\qquad\text{for every } g\in \Gamma.
\]
\end{proposition}

\begin{proof}
By item (3) of Lemma \ref{lem:interporlated regular representations}, the image of the representation $\lambda_\eps$ contains a free non-abelian subgroup subgroup of $F_r$. In particular, the $C^*$-algebra generated by $\lambda_\eps$, and hence also by $\tilde{\lambda}_\eps$, contains a copy of $C^*_r(F_m)$ for some $m\geq 2$. Thus, $\tilde{\lambda}_\eps$ is non-amenable by \cite[Corollary 2.11]{bedos1995notes}. It then follows from Lemma~\ref{lem:tensor_permanence} that $\pi_\eps$
is non-amenable as well.

Alternatively, we prove directly that $\tilde{\lambda}_\eps$ is non-amenable, as follows. Otherwise, there are unit vectors $(\xi_n)_{n}\subset L^2(LF_r\bar\otimes LF_r)$ such that $\|\tilde{\lambda}_\eps(g)\xi_n\tilde\lambda_\eps(g)^*-\xi_n\|_2\rightarrow 0$, for every $g\in \Gamma$.
Hence, $\|\lambda_\eps(g)\xi_n\lambda_\eps(g)^*-\xi_n\|_2\rightarrow 0$, for every $g\in F_r$. The construction of $\lambda_\eps$ gives $\lambda_\eps(g_i^m)=\lambda(g_i^{m^2})$ and so $\|\lambda(g_i^{m^2})\xi_n\lambda(g_i^{m^2})^*-\xi_n\|_2\rightarrow 0$, for every $1\leq i\leq r$.
This contradicts the fact that the representation $\text{Ad}(\lambda(g))$ of $F_r$ on $L^2(LF_r\bar\otimes LF_r)$ is a multiple of the left regular representation.
\end{proof}

We will need a more precise version of the above in the case of amalgamated free products.

\begin{proposition}\label{prop:amalgam_tensor_perturb}
Let $\Gamma=\Gamma_1 *_H \Gamma_2$, and assume that $\Gamma$ admits a homomorphism $\phi:\Gamma\to F_r$ such that $\phi(\Gamma_i)$ contains at least two of the standard generators of $F_r$ for $i=1,2$. Fix a representation $\pi:\Gamma\to \cU(M)$, let $\eps>0$, and consider $\pi_\eps$ as above. Then
\[
\lim_{\eps\to 0}\|\pi_\eps(g)-\pi(g)\|_2=0
\qquad\text{for every } g\in \Gamma,
\]
and the restrictions $\pi_\eps|_{\Gamma_i}$ are non-amenable for \(i=1,2\).
\end{proposition}

\begin{proof}
The proof is obtained by repeating the argument of the previous proposition.
\end{proof}
\subsubsection{Mixing}

We have seen that $\pi_\eps$ enjoys improved regularity properties compared to $\pi$. We will also need to know that the passage $\pi\mapsto \pi_\eps$ does not enlarge the center of the von Neumann algebra it generates.

\begin{lemma}\label{lem:mixing-lem}
    We have
    \[
    \cZ(\pi_\eps(\Gamma)'') \subset \cZ(\pi(\Gamma)'') \bar\otimes 1.
    \]
\end{lemma}

\begin{proof}
    Let $x\in \cZ(\pi_\eps(\Gamma)'')\subset M\bar\otimes LF_r$, and write its Fourier expansion as
    \[
    x=\sum_{h\in F_r} a_h\otimes u_h,
    \]
    where $a_h\in L^2(M)$ and
    \[
    \sum_{h\in F_r}\|a_h\|_2^2<\infty.
    \]

    Let $m\in \N$ be as in Lemma \ref{lem:interporlated regular representations}. Since $x$ belongs to the center of $\pi_\eps(\Gamma)''$, it commutes with $\pi_\eps(g)^m$ for every $g\in \Gamma$. Assume $\phi(g)$ is one of the standard generators of $F_r$. Then
    \begin{equation*}
    \begin{split}
        \pi_\eps(g)^m &= (\pi(g)\otimes \lambda_\eps(\phi(g)))^m = \pi(g)^m\otimes \lambda_\eps(\phi(g))^m\\
        &= \pi(g)^m\otimes \lambda(\phi(g))^{m^2} = \pi(g)^m\otimes u_{\phi(g)^{m^2}},
    \end{split}
    \end{equation*}
    and therefore
    \[
    (\pi(g)^m\otimes u_{\phi(g)^{m^2}})\,x\,(\pi(g)^{-m}\otimes u_{\phi(g)^{-m^2}})=x.
    \]
    Substituting the Fourier expansion of $x$, this becomes
    \[
    \sum_{h\in F_r} \pi(g)^m a_h \pi(g)^{-m}\otimes u_{\phi(g)^{m^2} h \phi(g)^{-m^2}}
    =
    \sum_{h\in F_r} a_h\otimes u_h.
    \]
    Comparing Fourier coefficients, we obtain
    \[
    a_{\phi(g)^{m^2} h \phi(g)^{-m^2}}
    =
    \pi(g)^m a_h \pi(g)^{-m}
    \qquad\text{for all }g\in \Gamma,\ h\in F_r.
    \]
    In particular,
    \[
    \|a_{\phi(g)^{m^2} h \phi(g)^{-m^2}}\|_2=\|a_h\|_2.
    \]

    Now fix $h\neq e$. Choose $s\in F_r$, one of the standard generators of $F_r$ such that it does not commute with $h$, and let $g\in \Gamma$ satisfy $\phi(g)=s$. We claim that the elements
    \[
    s^{nm^2}hs^{-nm^2},\qquad n\in \N,
    \]
    are pairwise distinct. Indeed, if
    \[
    s^{nm^2}hs^{-nm^2}=s^{\ell m^2}hs^{-\ell m^2}
    \]
    for some $n\neq \ell$, then
    \[
    s^{(n-\ell)m^2}h=hs^{(n-\ell)m^2},
    \]
    so $s^{(n-\ell)m^2}\in C_{F_r}(h)$. Since for every nontrivial element in a free group, there is always a largest cyclic subgroup containing it, and this largest cyclic subgroup is exactly the centralizer of the said element, this implies that $s\in C_{F_r}(h)$, contradicting the choice of $s$. Thus the orbit
    \[
    \{s^{nm^2}hs^{-nm^2}:n\in \N\}
    \]
    is infinite.

    Since
    \[
    \|a_{s^{nm^2}hs^{-nm^2}}\|_2=\|a_h\|_2
    \qquad\text{for all }n\in \N,
    \]
    it follows that if $a_h\neq 0$ for some $h\neq e$, then infinitely many Fourier coefficients of $x$ have the same positive $L^2$-norm, contradicting the summability condition
    \[
    \sum_{t\in F_r}\|a_t\|_2^2<\infty.
    \]
    Hence $a_h=0$ for every $h\neq e$, and so
    \[
    x=a_e\otimes 1\in M\bar\otimes 1.
    \]
    This shows that $\cZ(\pi_\eps(\Gamma)'')\subset M\bar\otimes 1$.

    Finally, since $x$ commutes with $\pi_\eps(g)=\pi(g)\otimes \lambda_\eps(\phi(g))$ for every $g\in \Gamma$, we get
    \[
    (a_e\otimes 1)(\pi(g)\otimes \lambda_\eps(\phi(g)))
    =
    (\pi(g)\otimes \lambda_\eps(\phi(g)))(a_e\otimes 1),
    \]
    hence
    \[
    a_e\pi(g)=\pi(g)a_e
    \qquad\text{for all }g\in \Gamma.
    \]
    Thus $a_e\in \pi(\Gamma)'\cap M=\cZ(\pi(\Gamma)'')$, and therefore
    \[
    x\in \cZ(\pi(\Gamma)'')\bar\otimes 1.
    \]
\end{proof}

\subsection{Perturbation II: free conjugation} This following perturbation is quite standard, and already appears in \cite{ioana2024trace} as a central method for addressing problems of this sort.

\begin{definition}
    Consider a tracial von Neumann algebra $M$ and subalgebras $M_1, M_2, N$ with $N \subset M_1 \cap M_2$. The tuple $(M, M_1, M_2, N)$ is said to be \emph{approximately factorial}, if for any $\eps > 0$ there exist $\widetilde M \supset M$ and a unitary $u \in N' \cap \widetilde M$ with $\|u - 1\|_2 < \eps$, such that $(uM_1u^{-1}) \vee M_2$ is a II$_1$ factor.
\end{definition}

\begin{lemma}\label{approx-factorial}
    Suppose that $\cZ(M) \cap N = \C$. Assume further that $M_i \nprec_M N$ holds for $i = 1, 2$. Then $(M, M_1, M_2, N)$ is approximately factorial.
\end{lemma}
\begin{remark}
    A reader unfamiliar with the meaning of \(M_i \nprec_M N\)  \cite{Po06} may simply replace this assumption with the requirement that \(N\) is of type I, whereas \(M_i\) has no type I summands.
\end{remark}
\begin{proof}[Proof of Lemma \ref{approx-factorial}]
    Set
    \begin{equation*}
        \widetilde M = M \ast_N (N \bar\otimes L\Z).
    \end{equation*}
    
    Let $u \in \cU(L\Z) \setminus \C$ with $\|u - 1\|_2 < \eps$ and set $P = (uM_1u^{-1}) \vee M_2$. Clearly, $u \in N'\cap \widetilde M$. The assumption $M_2 \nprec_M N$ implies by Theorem 1.1 of \cite{IPP08} that $M_2' \cap \widetilde M \subset M$, which in turn gives $\cZ(P) \subset M$. Similarly, we have $M_1' \cap \widetilde M \subset M$, implying $(uM_1u^{-1})' \cap \widetilde M \subset uMu^{-1}$, which in turn gives $\cZ(P) \subset uMu^{-1}$. Thus, $\cZ(P) \subset M \cap uMu^{-1}$. We note that $M \cap uMu^{-1} = N$. Indeed, since $N=uNu^{-1}\subset uMu^{-1}$, we have that $N \subset M \cap uMu^{-1}$. 
    Since $E_N(u)=\tau(u)$ and $|\tau(u)|<1$, Lemma \ref{intersect} implies that  $M \cap uMu^{-1}\subset N$.
    Therefore, $M \cap uMu^{-1} = N$ and thus $\cZ(P) \subset N$.

    We now claim that $\cZ(P) \subset \cZ(M)$ as well. Indeed, given $z \in \cZ(P)$, then $z \in M_2'$. Moreover, $z \in (uM_1u^{-1})'$ so $u^{-1}zu \in M_1'$. Since $z \in N$, we get that $z = u^{-1}zu \in M_1'$. Together, we get that $z \in M'$. But also $z \in N \subset M$, so $z \in \cZ(M)$. This shows the claim.

    Finally, we have assumed that $\cZ(M) \cap N = \C$. Hence, $\cZ(P) \subset \cZ(M) \cap N = \C$ meaning that $P$ is a factor. It is moreover infinite-dimensional because it contains $M_2$, and $M_2 \nprec_M N$.
\end{proof}

It will be useful for later use to record an easy particular 
instance which follows from the proof above.
\begin{lemma}\label{freely-conj-factor}
    If  $x$ and $y$ are two freely independent diffuse unitaries  in a tracial von Neumann algebra $M$, then the von Neumann algebra generated by $x$ and $yxy^*$ is a factor. 
\end{lemma}
\begin{proof}
    Consider the previous Lemma with $M=M_1=M_2 =\{x\}''$, $N=\C$ and $u=y$, and follow the argument in the first paragraph of the proof.
\end{proof}
Now, any $*$-representation $\pi$ of a full amalgamated free product of $C^*$-algebras $A=A_1*_B A_2$ 
into a tracial von Neumann algebra
gives rise to a tuple $(M,M_1,M_2,N)$ where $M=\pi(A)''$, $M_i=\pi(A_i)''$ and $N=\pi(B)''$.   Let
\[
    \widetilde M=M*_N(N \bar\otimes L\Z)
\]
and fix a unitary  $u\in L\Z \setminus \C $. Define $\widetilde \pi:A\to \widetilde M$ by $\widetilde \pi\mid _{A_1}= \pi\mid _{A_1}$ and $\widetilde \pi\mid _{A_2}=\Ad(u)\circ \pi\mid_{A_2}$. Note that by construction, $u$ commutes with $N$. One thus easily verifies the following:
\begin{lemma}
$\widetilde \pi$  is a well-defined representation of $A=A_1*_B A_2$. Moreover,  for any subset $S\subset A_1\cup A_2$ with $\|x\|\leq 1$, for every $x\in S$, we have 
\[
    \max_{x\in S}\|\widetilde \pi(x)-\pi(x)\|_2\leq 2\|u-1\|_2.
\]
\end{lemma}
One typically chooses $v$ to be a Haar unitary, and lets $u_t=e^{iht}$, where $h$ is a choice of logarithm of $v$. But in our context, it is enough to require $\|u-1\|_2$ to be small in order to ensure that $\widetilde \pi$ and $\pi$ are close. 

\subsection{Perturbation III: free multiplicative shift}

In order to apply the free conjugation perturbation from the last subsection to obtain a factor, we need, in particular, that $\cZ(M) \cap N = \C$. In our case, let $\Gamma_g$ be the fundamental group of a closed orientable surface of genus $g \geq 2$, and let $a_1, b_1, \cdots, a_g, b_g$ be the standard generators of $\Gamma_g$. Fix a tracial representation $\pi: \Gamma_g \to \cU(M)$. Our goal is to find a perturbed representation $\widetilde \pi: \Gamma_g \to \cU(\widetilde M)$ such that $\cZ(\widetilde \pi(\Gamma_g)'') \cap \widetilde \pi([a_1, b_1])'' = \C$.

In order to construct $\widetilde\pi$ for which we have better control over $\cZ(\widetilde \pi(\Gamma_g)'')$, we notice the following: if $u$ is a unitary in some algebra $\widetilde M$ containing $M$ which commutes with $\pi(a_1)$, then it is easy to check that
\begin{equation*}
    [\pi(a_1), \pi(b_1)u] = [\pi(a_1), \pi(b_1)].
\end{equation*}

Thus, $\widetilde\pi(a_i) = \pi(a_i)$ for $i = 1, \cdots, g$, $\widetilde\pi(b_1) = \pi(b_1)u$, and $\widetilde\pi(b_i) = \pi(b_i)$ for $i = 2, \cdots, g$ define a representation $\widetilde\pi: \Gamma_g \to \cU(\widetilde M)$. Then
\[\cZ(\widetilde \pi(\Gamma_g)'') \cap \widetilde \pi([a_1, b_1])'' = \cZ(\widetilde \pi(\Gamma_g)'') \cap \pi([a_1, b_1])'' \subset \{\pi(b_1)u\}' \cap M,\]
so it suffices to control the latter algebra.

\begin{lemma}\label{commutator-lem}
    Let $M$ be a tracial von Neumann algebra and $a, b \in \cU(M)$. Set $\widetilde M = M \ast_{\{a\}''} (\{a\}'' \bar\otimes L\Z)$ and $u \in \cU(L\Z) \setminus \C$. We then have $\{bu\}' \cap M \subset \{a\}''$.
\end{lemma}

\begin{proof}
    If $x\in \{bu\}' \cap M$, then we have that $x=(bu)^{-1}x(bu)=u^{-1}(b^{-1}xb)u$. Since $x,b\in M$, we get that $x\in M\cap u^{-1}Mu$. By using Lemma \ref{intersect} as in the proof of Lemma \ref{approx-factorial}, we get that $ M\cap u^{-1}Mu=\{a\}''$ and therefore $x\in\{a\}''$.
\end{proof}

\section{Concluding the proof of Theorem \ref{thm:main-tracial reps}}

\subsection{The surface group $\Gamma_g$ case}

We first prove Theorem \ref{thm:main-tracial reps} in the non-punctured case, i.e., we are working with the surface group $\Gamma_g$ for $g \geq 2$. The standard presentation of $\Gamma_g$ is given by
\begin{equation*}
    \langle a_1, a_2, \cdots, a_g, b_g| [a_1, b_1]\cdots[a_g, b_g]\rangle.
\end{equation*}

We note that it has a amalgamated free product decomposition given by $\Gamma_g = F_2 \ast_\Z F_{2g-2}$, where $F_2 = \langle a_1, b_1 \rangle$ and $F_{2g-2} = \langle a_i, b_i: i = 2, \cdots, g \rangle$ with the amalgam being generated by $[a_1,b_1] = ([a_2,b_2]\cdots[a_g,b_g])^{-1}$.

\begin{proof}[Proof of Theorem \ref{thm:main-tracial reps} in the surface group case]
    We start by considering an arbitrary tracial representation $\pi: \Gamma_g \to \cU(M)$. We apply the perturbation in subsection \ref{subsection 3.1} to control $\{\pi(a_1)\}'' \cap \{\pi([a_1,b_1])\}''$ first:
    
    \begin{claim}\label{conclusion-I}
        For any $\varepsilon > 0$, there exists a tracial von Neumann algebra $\widetilde M$ containing $M$ and a representation $\widetilde\pi: \Gamma_g \to \cU(\widetilde M)$ such that:
        \begin{enumerate}
            \item $\|\widetilde\pi(h) - \pi(h)\|_2 < \varepsilon$ for every $h = a_i$ or $b_i$, $i = 1, \cdots, g$;
            \item $\{\widetilde\pi(a_1)\}'' \cap \{\widetilde\pi([a_1, b_1])\}'' = \C$.
        \end{enumerate} 
    \end{claim}
    \begin{proof}[Proof of claim]
        We surject $\Gamma_g$ onto $\F_2 = \langle a, b \rangle$ via the quotient map $\phi$ defined by sending $a_1$ to $a$, $b_1$ to $b$, $a_2$ to $b$, $b_2$ to $a$, and $a_i$, $b_i$ to $e$ for $i > 2$. Consider the representation $\lambda_\varepsilon: \F_2 \to \cU(L(\F_2))$, $\tilde{\lambda}_\eps = \lambda_\varepsilon \circ \phi$ and then $\widetilde\pi = \pi \otimes \tilde{\lambda}_\eps$. That $\widetilde\pi$ satisfies condition 1 is clear. Now, note that,
    \begin{equation*}
    \begin{split}
        \{\tilde{\lambda}_\eps(a_1)\}'' \cap \{\tilde{\lambda}_\eps([a_1, b_1])\}'' &\subset \cZ(\{\lambda_\varepsilon(a), \lambda_\varepsilon([a, b])\}'')\\
        &= \cZ(\{\lambda_\varepsilon(a), \lambda_\varepsilon(b)\lambda_\varepsilon(a)\lambda_\varepsilon(b)^\ast\}'')\\
        &= \C,
    \end{split}
    \end{equation*}
    where the last equality is due to Lemma \ref{freely-conj-factor} and the observation that $\lambda_\varepsilon(a)$ and $\lambda_\varepsilon(b)$ are freely independent diffuse unitaries. Thus,
    \begin{equation*}
    \begin{split}
        \{\widetilde\pi(a_1)\}'' \cap \{\widetilde\pi([a_1, b_1])\}'' &\subset (\{\pi(a_1)\}'' \cap \{\pi([a_1, b_1])\}'') \bar\otimes (\{\tilde{\lambda}_\eps(a_1)\}'' \cap \{\tilde{\lambda}_\eps([a_1, b_1])\}'')\\
        &\subset \{\pi(a_1)\}'' \bar\otimes 1.
    \end{split}
    \end{equation*}
    
    Hence, it suffices to show $\{\pi(a_1) \otimes \lambda_\varepsilon(a)\}'' \cap (\{\pi(a_1)\}'' \bar\otimes 1) = \C$. Since the spectral measure of $\lambda_\varepsilon(a)$ is equivalent to the Haar measure of $\mathbb T$, there is a (non-trace-preserving) isomorphism $\{\lambda_\varepsilon(a)\}'' \to L\Z$ that sends $\lambda_\varepsilon(a)$ to the standard generator $u$ of $L\Z$. Since tensor products do not depend on the choice of traces, the isomorphism extends to an isomorphism $M \bar\otimes \{\lambda_\varepsilon(a)\}'' \to M \bar\otimes L\Z$ that is identity on $M \bar\otimes 1$ and sends $1\otimes\lambda_\varepsilon(a)$ to $1\otimes u$. Thus, we see that it suffices to show $$\{\pi(a_1) \otimes u\}'' \cap (\{\pi(a_1)\}'' \bar\otimes 1) = \C.$$ 
    As $\pi(a_1) \otimes u$ is a Haar unitary, this easily follows by writing out the Fourier series.   Indeed, given an element $x$ in this intersection, there is $(s_k) \in \ell^2$ such that 
    \[
        \sum_{k\in \Z} s_k \pi(a_1)^k\otimes u^k =x.
    \]\
    
    Hence, as $x \in M \otimes 1$, $(\id \otimes \tau)(x) = x$, and thus
    \[
        s_0 = (\id \otimes \tau)\left(\sum_{k\in \Z} s_k \pi(a_1)^k\otimes u^k\right) = x,
    \]
    i.e., $x\in \C$.
    \end{proof}

    The above claim shows the representations $\pi$ for which $\{\pi(a_1)\}'' \cap \{\pi([a_1,b_1])\}'' = \C$ are dense in the Wasserstein topology. Hence, we may assume $\pi$ satisfies this condition. We now apply Lemma \ref{commutator-lem}. Choose $u \in \cU(L\Z) \setminus \C$ in $\widetilde M = M \ast_{\{\pi(a_1)\}''} (\{\pi(a_1)\}'' \bar\otimes L\Z)$. Then $\widetilde\pi(a_i) = \pi(a_i)$ for $i = 1, \cdots, g$, $\widetilde\pi(b_1) = \pi(b_1)u$, and $\widetilde\pi(b_i) = \pi(b_i)$ for $i = 2, \cdots, g$ define a representation $\widetilde\pi: \Gamma_g \to \cU(\widetilde M)$. By Lemma \ref{commutator-lem}, we obtain
    \begin{equation*}
    \begin{split}
        \cZ(\widetilde\pi(\Gamma_g)'') \cap \{\widetilde\pi([a_1, b_1])\}'' &\subset \{\pi(b_1)u\}' \cap \{\pi([a_1, b_1])\}''\\
        &\subset \{\pi(b_1)u\}' \cap M \cap \{\pi([a_1, b_1])\}''\\
        &\subset \{\pi(a_1)\}'' \cap \{\pi([a_1, b_1])\}''\\
        &= \C.
    \end{split}
    \end{equation*}
    
    Since we may choose $u$ to be arbitrarily close to 1 in $2$-norm, this implies that the representations $\pi$ for which $\cZ(\pi(\Gamma_g)'') \cap \{\pi([a_1, b_1])\}'' = \C$ are dense in the Wasserstein topology, so we shall now assume $\pi$ satisfies this condition.

    We now tensor our representation with $\tilde{\lambda}_\eps$ defined in the proof of Claim \ref{conclusion-I} again. By Lemma \ref{lem:mixing-lem}, we have
    \begin{equation*}
    \begin{split}
        \cZ(\widetilde\pi(\Gamma_g)'') \cap \{\widetilde\pi([a_1, b_1])\}'' &\subset (\cZ(\pi(\Gamma_g)'') \bar\otimes 1) \cap \{\pi([a_1, b_1]) \otimes \tilde{\lambda}_\eps([a_1, b_1])\}''\\
        &\subset (\cZ(\pi(\Gamma_g)'') \cap \{\pi([a_1, b_1])\}'') \bar\otimes 1\\
        &= \C.
    \end{split}
    \end{equation*}

    On the other hand, we also note that, by Proposition \ref{prop:amalgam_tensor_perturb}, we have $\widetilde\pi|_{\langle a_1, b_1 \rangle}$ is non-amenable and thus, by Lemma \ref{lem:no_amenable_summand}, $\widetilde\pi(\langle a_1, b_1 \rangle)''$ has no amenable summands. Similarly, $\widetilde\pi|_{\langle a_i, b_i: i = 2, \cdots, g \rangle}$ is non-amenable and so $\widetilde\pi(\langle a_i, b_i: i = 2, \cdots, g \rangle)''$ has no amenable summands. That is:

    \begin{claim}\label{conclusion-III}
        For any $\varepsilon > 0$, there exists a tracial von Neumann algebra $\widetilde M$ containing $M$ and a representation $\widetilde\pi: \Gamma_g \to \cU(\widetilde M)$ such that:
        \begin{enumerate}
            \item $\|\widetilde\pi(h) - \pi(h)\|_2 < \varepsilon$ for every $h = a_i$ or $b_i$, $i = 1, \cdots, g$;
            \item $\cZ(\widetilde\pi(\Gamma_g)'') \cap \{\widetilde\pi([a_1, b_1])\}'' = \C$;
            \item Both $\widetilde\pi|_{\langle a_1, b_1 \rangle}$ and $\widetilde\pi|_{\langle a_i, b_i: i = 2, \cdots, g \rangle}$ are non-amenable. Hence, $\widetilde\pi(\langle a_1, b_1 \rangle)'' \nprec_{\widetilde M} \{\widetilde\pi([a_1, b_1])\}''$ and $\widetilde\pi(\langle a_i, b_i: i = 2, \cdots, g \rangle)'' \nprec_{\widetilde M} \{\widetilde\pi([a_1, b_1])\}''$.
        \end{enumerate}
    \end{claim}

    Finally, we apply Lemma \ref{approx-factorial} to the tuple
    \begin{equation*}
        (\widetilde\pi(\Gamma_g)'', \widetilde\pi(\langle a_1, b_1 \rangle)'', \widetilde\pi(\langle a_i, b_i: i = 2, \cdots, g \rangle)'', \{\widetilde\pi([a_1, b_1])\}'').
    \end{equation*}
    
    This generates a factorial representation $\widetilde{\widetilde\pi}$. It is also non-amenable since the perturbation does not affect the representation when restricted to $\langle a_i, b_i: i = 2, \cdots, g \rangle$. Since $\widetilde\pi|_{\langle a_i, b_i: i = 2, \cdots, g \rangle}$ is non-amenable, so is $\widetilde{\widetilde\pi}|_{\langle a_i, b_i: i = 2, \cdots, g \rangle}$. Whence $\widetilde{\widetilde\pi}$ is non-amenable. This concludes the proof of Theorem \ref{thm:main-tracial reps} in the surface group case.
\end{proof}

\subsection{The free group $F_g$ case}

The proof for the free group case is substantially easier and uses much less than established in general. For this reason, we will make this proof somewhat self-contained, while still referencing other parts when necessary. We will write the standard free generators of $F_g$ as $a_1, \cdots, a_g$.

\begin{proof}[Proof of Theorem \ref{thm:main-tracial reps} in the free group case]
    The first part here is to perturb the representation so that $\pi(a_i)$ is diffuse for all $i$. The same perturbation will also be used again to ensure non-amenability. This is given by Lemma \ref{lem:interporlated regular representations} and Proposition \ref{prop:free_quotient_tensor_perturb}. But in the case of the free group we can use a simpler version. Indeed, consider the regular representation $\lambda:F_g\to\cU(\ell^2(F_g))$. We have trace preserving isomorphism of von Neumann algebras $\{\lambda(a_i)\}''\cong L^\infty(\T,\mu_{\mathrm{Leb}})$ sending $\lambda(a_i)$ to the functions $f:e^{2\pi i}\mapsto e^{2\pi i}$. For $m\in \N$ consider the function $f^{1/m}:e^{2\pi it}\mapsto e^{2\pi it/m}$, and consider the corresponding element of $LF_g$, namely, the element $\lambda_m(a_i)=f^{1/m}(\lambda(a_i))$. We note that $\lambda_m(a_i)$ is diffuse. Moreover, the representation $\lambda_m$ is non-amenable, because when restricted to $\langle a_1^m, \cdots, a_g^m\rangle \cong F_g$, the representation is simply the regular representation of $F_g$, which is non-amenable. Also, it is clear that $\lambda_m(a_i)\to 1$ in operator norm, so in particular in the $\|\cdot\|_2$ norm as well. Hence, $\lambda_m\to 1$ in the Wasserstein topology. 

    Now, given an arbitrary tracial representation $\pi$, we consider $\widetilde{\pi}_m=\pi\otimes \lambda_m$, so that $\widetilde{\pi}_m\to \pi$ in the Wasserstein topology. The fact that $\lambda_m$ is non-amenable implies that $\widetilde{\pi}_m$ is non-amenable by Lemma \ref{lem:tensor_permanence}. Since $\lambda_m(a_i)$ is diffuse for all $i$, $\widetilde\pi_m(a_i)$ is diffuse as well. Thus, we may start by assuming $\pi(a_i)$ is diffuse for all $i$.

    To that end, let $\pi$ be an arbitrary tracial representation in which $\pi(a_i)$ is diffuse for all $i$. Let $M$ be the corresponding von Neumann algebra. Consider the free product $\widetilde M = M*L\Z$, and let $u\in L\Z \setminus \C$ unitary with $\|u-1\|_2<\eps$. Let $M_1=\pi(F_{g-1})''$ and  $M_2=\pi(\Z)''$ where we write $F_{g-1}=\langle a_1, \cdots, a_{g-1}\rangle$ and $\Z = \langle a_g\rangle$. Set $P=uM_1u^*\vee M_2$. The proof of Lemma \ref{approx-factorial} shows $P$ is a factor. We summarize the relevant argument here: since $\pi(a_i)$ is diffuse for all $i$, $M_1$ and $M_2$ are both diffuse, whence $M_k \not\prec_M \C$ for $k = 1, 2$. By Theorem 1.1 of \cite{IPP08}, $M_2' \cap \widetilde M \subset M$. Similarly, $(uM_1u^\ast)' \cap \widetilde M \subset uMu^\ast$. The result now follows from Lemma \ref{intersect}.

    This implies the set of tracial factorial representations is dense in the Wasserstein topology. Thus, let $\pi$ be an arbitrary tracial factorial representation. As discussed previously, we may consider $\widetilde{\pi}_m=\pi\otimes \lambda_m$, which is non-amenable. It therefore suffices to show it is factorial. This follows from the same proof as that of Lemma \ref{lem:mixing-lem}.
\end{proof}

\bibliographystyle{amsalpha}
\bibliography{references}

\end{document}